\documentclass[12pt]{amsart}
\usepackage{color}
\usepackage{url}
\usepackage{hyperref}
\usepackage{amsmath}
\usepackage{amsfonts}
\usepackage{amssymb}
\usepackage{amsthm}
\usepackage{geometry}
\usepackage{enumerate}
\ifdefined\directlua
\usepackage{polyglossia}
\setmainlanguage{english}

\usepackage[draft]{todonotes}

\usepackage{fontspec}
\usepackage{graphics}
\usepackage{microtype}
\usepackage{float}
\restylefloat{table}
\usepackage{pdftexcmds}
\makeatletter
\ifcase\pdf@shellescape
\relax\or
\usepackage{luacode}
\begin{luacode}
  function prtgit()
  local cmd="git show -s --format='%ci %h'"
  print(cmd)
  local r=io.popen(cmd):read("*a")
  if (r) then
  tex.print([[\string\def\string\COMMIT{]]..r..[[}]])
  end
  end
\end{luacode}
\directlua{prtgit()}
\or
\relax\fi
\makeatother
\ifdefined\COMMIT
        \usepackage{background}
        \backgroundsetup{%
         pages=all, placement=bottom,angle=0,scale=1.6,%
         vshift=20pt,hshift=0pt,
         contents={Commit version:\COMMIT}}
\fi
\fi

\newcommand{\F}{\mathbb F}
\newcommand{\cX}{\mathcal X}
\newcommand{\cA}{\mathcal A}
\newcommand{\cW}{\mathcal W}
\def\cI{\mathcal I}

\newcommand{\C}{\mathcal C}
\newcommand{\bC}{\mathbb C}
\newcommand{\bF}{\mathbb F}

\newcommand{\PP}{\mathbb{P}}
\newcommand{\bP}{\mathbb{P}}

\newcommand{\cP}{\mathcal P}

\newcommand{\VV}{\mathbb V}

\newcommand{\cC}{\mathcal C}

\newcommand{\PG}{\mathrm{PG}}
\newcommand{\PGL}{\mathrm{PGL}}

\newcommand{\Aut}{\mathrm{Aut}\,}

\newtheorem{theorem}{Theorem}[section]
\newtheorem{lemma}[theorem]{Lemma}
\newtheorem{corollary}[theorem]{Corollary}
\newtheorem{proposition}[theorem]{Proposition}

\theoremstyle{definition}
\newtheorem{remark}[theorem]{Remark}
\newtheorem{example}[theorem]{Example}

\newtheorem{definition}{Definition}

\newenvironment{nouppercase}{%
  \renewcommand{\uppercasenonmath}[1]{}}{}

\title{Waring identifiable subspaces over finite fields}
\author{Michel Lavrauw and Ferdinando Zullo}
\thanks{
The first author acknowledges the support of {\em The Scientific and Technological Research Council of Turkey} T\"UB\.{I}TAK (project no.~118F159).\\
The second author was supported by the project ``VALERE: VAnviteLli pEr la RicErca" of the University of Campania ``Luigi Vanvitelli'' and by the Italian National Group for Algebraic and Geometric Structures and their Applications (GNSAGA - INdAM)}
\begin{document}
\begin{nouppercase}
\maketitle
\end{nouppercase}

\begin{abstract}
Waring's problem, of expressing an integer as the sum of powers, has a very long history going back to the 17th century, and the problem has been studied in many different contexts.
In this paper we introduce the notion of a {\it Waring subspace} and a {\it Waring identifiable subspace} with respect to a projective algebraic variety $\cX$. When $\cX$ is the Veronese variety, these subspaces play a fundamental role in the theory of symmetric tensors and are related to the Waring decomposition and Waring identifiability of symmetric tensors (homogeneous polynomials). We give several constructions and classification results of Waring identifiable subspaces with respect to the Veronese variety in $\PP^5(\bF_q)$ and in $\PP^{9}(\bF_q)$, and include some applications to the theory of linear systems of quadrics in $\PP^3(\bF_q)$.
\end{abstract}



\section{Introduction}

Determining the rank and the decomposition of tensors is an important and notoriously difficult problem. A special case of this is one of the classical problems for symmetric tensors (multilinear forms): determine the minimum integer $k$ such that a generic symmetric tensor $f\in \mathrm{Sym}^d(V)$ can be written as the sum of $k$ pure tensors of $\mathrm{Sym}^d(V)$. This problem is a reformulation of writing a homogeneous polynomial $f$ of degree $d$ as the sum of $d$-th powers of linear forms, and can be seen as a generalisation of the number theory problem posed by Waring in \cite{Waring}. The connection is given by the correspondence between homogeneous polynomials of  degree $d$ in $\mathbb{F}[X_0,\ldots,X_n]$ and the elements of $\mathrm{Sym}^d(V)$ through so-called {\it polarization}.
The value $k$ is called the \emph{Waring rank of $f$} and the decomposition of $f$ into the sum of $k$ pure tensors is called a {\it Waring decomposition}. 
If the linear forms appearing in a minimal decomposition are unique, up to a nonzero scalar multiple, then $f$ is called \emph{Waring identifiable}. The question of identifiability is naturally interesting on its own and has many applications. When $\bF$ is the field of complex numbers, the Waring rank of a {\it generic} form in $\mathrm{Sym}^d(\bC^{n+1})$ was determined by Alexander and Hirschowitz in \cite{AlHi1995}. Three decades later Chiantini, Ottaviani and Vanniewenhoven \cite{ChOtVa2017} showed that in all but three exceptions, a general form in $\mathrm{Sym}^d(\bC^{n+1})$ of subgeneric rank is Waring identifiable, extending results from Ranestad and Voisin \cite{RaVo2017}, and Ballico \cite{Ballico2005}.

The notions introduced above can be generalized to $r$-tuples of forms in 
$$\mathbb{F}[X_0,\ldots,X_n]_{d_1}\times \ldots \times \mathbb{F}[X_0,\ldots,X_n]_{d_r}$$ 
in which case one speaks of \emph{simultaneous Waring decomposition}, \emph{Waring rank} and \emph{Waring identifiability}. 
When $\mathbb{F}\in \{\mathbb{R},\mathbb{C}\}$, a well studied problem is to find the parameters $(r,n,d_1,\ldots,d_r,k)$ for which a {\it generic} $r$-tuple of forms is Waring identifiable. For more on this topic (and further references), we refer the interested reader to the recent papers \cite{AGMO} and \cite{AnCh2020}.

\bigskip

The aim of this paper is to pursue the study of Waring identifiability over finite fields, in particular for $\mathrm{Sym}^d(V)$, where $V$ is a finite-dimensional vector space over a finite field $\bF=\bF_q$.
Note that this is related, but different, from the problem of determining the tensor rank (or symmetric tensor rank) of tensors over finite fields (see e.g. \cite{Laskowski1982}, \cite{ChCh1988}, \cite{LaPaZa2013}, \cite{CeOz2010}, \cite{LaSh2022}, \cite{SoZhHu2022} and references therein), which is related to complexity theory, and where one is usually interested in either bounding the possible number of terms of a minimal decomposition of a tensor or finding a decomposition.
For our approach, we introduce the notion of a {\it Waring subspace} and the notion of a {\it Waring identifiable} subspace.
These can be seen as extensions of the notions of \emph{decomposition} and \emph{identifiability} given in \cite{BBCC}
by Ballico, Bernardi,  Catalisano and  Chiantini for any irreducible non-degenerate projective variety over an algebraically closed field of characteristic zero. These notions will be defined in Section \ref{sec:preliminaries} in a more general context, namely with respect to any algebraic variety $\cX$, or any set of points $\cA$. In the particular case of $\cX$ being equal to the Veronese variety, a Waring subspace corresponds to a subspace of $\mathrm{Sym}^d(V)$ spanned by pure symmetric tensors. 

Besides the intrinsic theoretical interest, the study of Waring subspaces is also motivated by applications in coding theory, see e.g. \cite{code1,code3,Aubry1992}.
For instance in \cite{code1} Edoukou studied the functional codes associated with quadrics in $\mathbb{P}^3(\mathbb{F}_q)$. For a fixed quadric $\cX$ in $\bP^3(\bF_q)$, with $|\cX(\bF_q)|=n$, the codewords of the code $\C_2(\cX)$ are the vectors of $\bF_q^n$ obtained by evaluating quadratic forms on $\bP^3(\bF_q)$ at the $\bF_q$-rational points of $\cX$. The weight of a codeword $c(f)$ of $\C_2(\cX)$ is determined by the size of intersection of $\cX(\bF_q)$ with the quadric defined by $f$. So in order to determine the parameters of these codes one needs to understand pencils of quadrics, and in particular the size of the base of these pencils. The base of such a pencil corresponds to the intersection of a Veronese variety $\VV_{3,2}$ with a subspace of co-dimension two in $\bP^9(\bF_q)$.
As we will see in the next sections, this is the kind of problems which are studied in this paper.

\bigskip

\begin{remark}
We note that a different version of Waring's problem over finite fields, where one is interested in determining the minimum number $g(k,p^n)$ such that every element of the finite field $\bF_{p^n}$ of order $p^n$ can be written as the sum of at most $g(k,p^n)$ $k$-th powers of elements in $\bF_{p^n}$ has been well-studied since the 1970's, see for example Dodson \cite{Dodson1971}, Dodson and Tiet\"av\"ainen \cite{DoTi1976}, Konyagin \cite{Konyagin1992} and Winterhof \cite{Winterhof1998}.
\end{remark}

\bigskip

The paper is organized as follows. 
In Section \ref{sec:Vero} we briefly recall properties of symmetric tensors and the links with the Veronese variety $\VV_{n,d}$.
In Section \ref{sec:Waring} we investigate the Waring identifiable subspaces with respect to the Veronese variety of degree two, including methods to construct examples of such subspaces.
A complete classification of identifiable Waring subspaces with respect to $\VV_{n,d}({\F_q})$ when $d=n=2$ is given in Section \ref{sec:2,2},
and in Section \ref{sec:3,2} we describe those for $d=2$, $n=3$ and $\mathbb{F}=\bF_q$ classifying Waring identifiable hyperplanes, and subspaces of co-dimension two (except for a finite number of values of $q$). We also prove the non-existence of Waring subspaces of ${\mathrm{Sym}}^2(\bF_q^4)$ of co-dimension two for $q\geq 53$.

\section{Preliminaries}\label{sec:preliminaries}

\subsection{Symmetric tensors and rank}
Consider the finite-dimensional vector spaces $V_1,\ldots,V_d$ defined over the field $\mathbb{F}$. The tensor product $V_1\otimes \ldots \otimes V_d$ is the vector space of all multilinear functions from $V_1^\vee \times \ldots \times V_d^\vee$ to $\mathbb{F}$, where $V_i^\vee$ denotes the dual of $V_i$. The {\it pure tensors} are the tensors denoted by $v_1\otimes \ldots \otimes v_d$ with $v_1\in V_1, \ldots, v_d\in V_d$, and defined by:
\begin{eqnarray}
v_1\otimes \ldots \otimes v_d~:~(u_1,\ldots,u_d)\mapsto \prod v_i(u_i).
\end{eqnarray}
The set of pure tensors of $V_1\otimes \ldots \otimes V_d$ span the whole space and the minimum number of pure tensors which are required to express a tensor $T\in V_1\otimes \ldots \otimes V_d$ as their sum is called {\it the rank of $T$} (whence the alternative terminology {\it tensors of rank one} for pure tensors).
If all the factors $V_1,\ldots,V_d$ of the tensor product $V_1\otimes \ldots \otimes V_d$ are equal to $V$, we write $\otimes^d V$ instead of $V\otimes \ldots \otimes V$ ($d$ times).
In this case one may consider the action on $\otimes^d V$ of  the symmetric group $S_d$ of degree $d$ by defining
\begin{eqnarray}\label{eqn:sigma}
\sigma (v_1\otimes \ldots\otimes v_d)= v_{\sigma(1)}\otimes \ldots \otimes v_{\sigma(d)}.
\end{eqnarray}
for $\sigma\in S_d$, on the set of pure tensors of $\otimes^d V$, and linearly extending this natural action (\ref{eqn:sigma}) of $S_d$ to the whole of $\otimes^d V$.
A tensor which is left invariant under this action of $S_d$ on $\otimes^d V$ is called a {\it symmetric tensor}. The symmetric tensors form a subspace of $\otimes^dV$ which we denote by $\mathrm{Sym}^d(V)$. A {\it symmetric pure tensor} (called a {\it polar element} in \cite{Pate})
$$v^d:=v\otimes \ldots \otimes v
$$
therefore corresponds to the multilinear map from $V^\vee\times \ldots \times V^\vee$ to $\mathbb{F}$ given by
\begin{eqnarray}
v^d(u_1,\ldots,u_d)=\prod_{i=1}^d v(u_i).
\end{eqnarray}

\subsection{Waring subspaces}
If $V$ is an $(n+1)$-dimensional vector space over a field $\mathbb{F}$ we denote by $\PP V $, $\PP^n(\mathbb{F})$, or simply by $\PP^n$ (if $\bF$ is clear from the context), or $\mathbb{P}^n_q$ if $\mathbb{F}=\mathbb{F}_q$, the projective space of dimension $n$ over $\bF$. 
An {\it algebraic variety} in $\PP^n$ is the zero locus of a homogeneous ideal in $\bF[X_0,\ldots,X_n]$. Note that some texts assume an algebraic variety to be an irreducible algebraic set, but we do not need this assumption here. 
For the purpose of this paper, an algebraic variety $\cX$ comes with a specified embedding in $\PP^n$. Moreover, we will consider algebraic varieties $\cX$ in $\PP^n$ satisfying the property that ${\cX}(\mathbb{F})$ spans $\PP^n$.
Rather than the study of algebraic varieties up to {\it birational equivalence}, we are interested in the properties of algebraic varieties up to {\it $\mathrm{PGL}$-} or {\it $\mathrm{P\Gamma L}$-equivalence}. 

\bigskip

Let $\cA$ be any set of points of $\PP^n(\bF)$, and let $\cX$ be an algebraic variety in $\PP^n(\bF)$. Given a subfield or extension field $\bF'$ of $\bF$, we denote by $\cX(\mathbb{F'})$ the set of $\mathbb{F'}$-rational points of $\cX$. 

\begin{definition}
A subspace $S$ of $\PP^n(\bF')$ is called a \emph{Waring subspace with respect to $\cX$} if it is  spanned by $\mathbb{F'}$-rational points of $\cX$.
A subspace $S$ of $\PP^n(\bF')$ is called a \emph{Waring subspace with respect to $\cA$} if it is  spanned by points of $\cA$.
The {\it witness} of a Waring subspace $U$ with respect to $\cX$ (respectively $\cA$) of $\PP^n(\bF')$ is the set $U\cap \cX(\bF')$ (respectively $U\cap \cA$).

Furthermore, a subspace (or a subset of subspaces) $S$ of $\PP^n(\bF')$ is called \emph{Waring identifiable with respect to} $\cX$ if its witness $W$ is unique and no proper subset of $W$ spans a subspace containing $S$. In other words, $S$ is contained in a unique Waring subspace $U$ of minimal dimension, satisfying the property that the set of $\mathbb{F'}$-rational points of $\cX$ spanning $U$ is unique. A subspace $S$ of $\PP^n(\bF')$ is called an {\it identifiable Waring subspace with respect to} $\cX$ if $S$ is a Waring subspace which is also Waring identifiable with respect to $\cX$. Similarly, we define a {\it Waring identifiable subspace} and an {\it identifiable Waring subspace} with respect to a set of points $\cA$.

\end{definition}

\begin{remark}
In \cite{BBCC}, for a fixed irreducible non-degenerate projective variety $\mathcal{X}$ contained in $\mathbb{P}^n(\mathbb{F})$, where $\mathbb{F}$ is an algebraically closed of characteristic zero, the authors introduced the notion of \emph{decomposition} and those of $\mathcal{X}$-\emph{identifiable} of a projective subspace $S$; see \cite[Definitions 1.4 and 1.5]{BBCC} .
The former coincides with the notion of witness of any Waring subspace (with respect to $\mathcal{X}$) containing $S$ and the latter coincides with the notion of being Waring identifiable with respect to $\cX$. In \cite{BBCC} the authors give a criterion for the identifiability of a generic subspace of dimension $k$ by interpreting this property in terms of the Segre product $\mathrm{Seg}(\mathbb{P}^k\times \mathcal{X})$, see \cite[Theorems 3.1 and 3.3]{BBCC}.
\end{remark}

If $\cX$ is clear from the context, we will simply speak of (identifiable) Waring subspaces and Waring identifiable subspaces (or sets of subspaces). The definition of Waring identifiability can also be formulated in terms of the $\cX$-rank.
Recall that if $S$ is a subspace of $\PP^n(\bF)$, then its $\cX$-\emph{rank} (over $\mathbb{F}$), notation $rk_{\cX(\bF)}(S)$, is the minimum number of points of $\cX(\mathbb{F})$ needed to span a subspace of $\PP^n(\bF)$ containing $S$.
It follows that a point $P$ in $\PP^n(\bF)$ is Waring identifiable if it is contained in exactly one Waring subspace of $\PP^n(\bF)$ of dimension the $\cX$-rank of $P$ minus one.
When the chosen algebraic variety is the Veronese variety, this definition is equivalent to the one given earlier on, related to the theory of symmetric tensors and in particular to Waring's problem as described in the introduction.

Clearly, a point of $\bP(\bF)$ is Waring identifiable (and also a Waring subspace) with respect to $\cX$ if and only if it is a point of $\cX(\bF)$.
A line  $\ell$ of $\bP(\bF)$ is a Waring subspace with respect to $\cX$ if and only if $|\ell \cap \cX(\bF)|\geq 2$, and if $|\ell \cap \cX(\bF)|= 2$ then $\ell$ is also Waring identifiable. Note that a Waring subspace is not necessarily Waring identifiable; a trivial example of a Waring subspace which is not Waring identifiable is given by a trisecant to a planar cubic curve.

\begin{example}\label{ex:conic}
For a non-degenerate conic $\cX$ in $\PP^2_q$, all of its points are identifiable Waring subspaces.
For a non-degenerate conic $\cX$ in the Fano plane $\PP^2_2$, all but one point not on $\cX$ have $\cX$-rank two and are Waring identifiable, the remaining point (which is the nuclei of $\cX$) is also Waring identifiable, but has $\cX$-rank $3$. The Waring lines are the secant lines, and they are identifiable Waring subspaces. The tangent lines and the external lines are Waring identifiable with respect to $\cX$ but are not Waring subspaces.
For a non-degenerate conic $\cX$ in $\PP^2_3$, all the exterior points to $\cX$ are Waring identifiable with respect to $\cX$. They have $\cX$-rank two. The interior points are not Waring identifiable. The Waring lines are the secant lines, whereas the remaining lines are neither Waring subspaces nor Waring identifiable.
For a non-degenerate conic $\cX$ in $\PP^2_q$ with $q\geq 4$, points off $\cX$ are neither Waring subspaces nor Waring identifiable. The Waring lines are the secant lines, whereas the remaining lines are neither Waring subspaces nor Waring identifiable.
\end{example}

We are interested in Waring (identifiable) subspaces up to the following natural equivalence relation. Let $\Aut(\cX(\bF))$ denote the group of collineations of $\bP^n(\bF)$ fixing $\cX(\bF)$. Note that $\Aut(\cX)$ is not necessarily equal to $\Aut(\cX(\bF))$, see e.g. \cite[Exercise 10.7]{Harris}. Two subspaces $S_1$ and $S_2$ of $\PP^n(\bF)$ are called 
$\cX$-\emph{equivalent} (or {\it equivalent with respect to $\cX$}) if there exists an automorphism $\varphi \in \Aut(\cX(\bF))$ such that $\varphi(S_1)=S_2$. If $\varphi$ is linear, then the subspaces $S_1$ and $S_2$ are called {\it projectively equivalent} ({\it with respect to $\cX$}).

We define the {\it Waring polynomial of $\cX$ over $\bF$} (or {\it $\bF$-Waring polynomial of $\cX$}) as
\begin{eqnarray}\label{eqn:Waring_pol}
\cW_\cX(\bF)=\sum_{i=0}^{n-1} \lambda_i(\cX(\bF)) X^i
\end{eqnarray}
where $\lambda_i(\cX(\bF))$ is the number of ${\mathrm{Aut}}(\cX(\bF))$-orbits of $i$-dimensional Waring subspaces of $\PP^n(\bF)$ under the action of $\Aut(\cX(\bF))$ on the set of subspaces of $\bP^n(\bF)$. Note that $\cW_\cX(\bF)$ is meaningful when $\Aut(\cX(\bF))$ has a finite number of orbits on subspaces of $\bP^n(\bF)$, a property which is obviously satisfied for $\bF=\bF_q$.

The {\it Waring identifiable polynomial} of $\cX$ is defined as
\begin{eqnarray}\label{eqn:Waring_pol}
\cW\cI_\cX(\bF)=\sum_{i=0}^{n-1} \mu_i(\cX(\bF)) X^i
\end{eqnarray}
where $\mu_i(\cX(\bF))$ is the number of ${\mathrm{Aut}}(\cX(\bF))$-orbits of $i$-dimensional Waring identifiable subspaces of $\PP^n(\bF)$.

The {\it identifiable Waring polynomial} of $\cX$ is defined as
\begin{eqnarray}\label{eqn:Waring_pol}
\cI\cW_\cX(\bF)=\sum_{i=0}^{n-1} \eta_i(\cX(\bF)) X^i
\end{eqnarray}
where $\eta_i(\cX(\bF))$ is the number of ${\mathrm{Aut}}(\cX(\bF))$-orbits of $i$-dimensional identifiable Waring subspaces of $\PP^n(\bF)$.

\begin{example}
The above defined polynomials are easily obtained for a non-degenerate conic $\cX$ in $\PP^2_q$. The analysis made in Example \ref{ex:conic} implies that for $q=2$ we have
\[ \mathcal{W}_{\cX}(\bF_2)=1+X,\,\,\mathcal{WI}_{\cX}(\bF_2)=3+3X,\,\, \mathcal{IW}_{\cX}(\bF_2)=1+X; \]
for $q=3$ we have
\[ \mathcal{W}_{\cX}(\bF_3)=1+X,\,\,\mathcal{WI}_{\cX}(\bF_3)=2+X,\,\, \mathcal{IW}_{\cX}(\bF_3)=1+X; \]
and for $q\geq 4$ we have
\[ \mathcal{W}_{\cX}(\bF_q)=1+X,\,\,\mathcal{WI}_{\cX}(\bF_q)=1+X,\,\, \mathcal{IW}_{\cX}(\bF_q)=1+X. \]
\end{example}

\subsection{The Veronese variety}\label{sec:Vero}

Since we are interested in the Waring problem and the (symmetric) rank of tensors in ${\mathrm{Sym}}^dV$, the algebraic variety at hand is the Veronese variety, which we now define.

Following \cite{Harris}, for any $n$ and $d$, we define the \emph{Veronese map of degree $d$}

\begin{equation*}
\begin{array}{cccc}
\nu_d\colon & \PP^n & \longrightarrow & \PP^N\\
& (X_0,\ldots,X_n) & \mapsto & (\ldots,X^I,\ldots)
\end{array}
\end{equation*}

\noindent where $X^I$ ranges over all monomials of degree $d$ in $X_0,\ldots,X_n$ and $N={{n+d}\choose{d}}-1$.
The image of the Veronese map is an algebraic variety called \emph{Veronese variety} and usually denoted by $\VV_{n,d}$ (or $\VV_{n,d}(\mathbb{F})$). The Veronese variety $\VV_{n,2}$ obtained for $d=2$ is also called the {\it quadric Veronesean}.
If we regard $\PP^N$ as $\PP(\mathbb{F}[x_0,\ldots,x_n]_d)$, i.e.\ as $\PP(\mathrm{Sym}^d(V))$, the Veronese variety parametrizes those polynomials that can be written as $d$-th powers of a linear form, i.e. the pure elements of $\mathrm{Sym}^d(V)$.
The $\VV_{n,d}$-rank of a point in $\PP(\mathrm{Sym}^d(V))$ coincides with the rank of the nonzero tensors defining the point, and a point is Waring identifiable with respect to $\VV_{n,d}$ if the nonzero tensors defining the point are Waring identifiable.

The projective stabliser of $\VV_{n,d}(\mathbb{F})$ is in general obtained by lifting the action of $\PGL(n+1,\mathbb{F})$ from $\PP^n$ to $\PP^N$ through the Veronese map $\nu_d$, see e.g. \cite{Harris}, however caution should be taken when working over $\mathbb{F}=\bF_2$, in which case the projective stabiliser of $\VV_{n,d}(\mathbb{F})$ is the full symmetric group $S_7$, instead of the much smaller ${\mathrm{PSL}}(3,2)$.

An important property of the Veronese map of degree $d$ is that hypersurfaces in $\bP^n$ of degree $d$ correspond to hyperplanes of $\bP^N$. Moreover, the set of $\bF$-rational points of a hypersurface of degree $d$ corresponds to the set of points in a hyperplane section of the Veronese variety $\VV_{n,d}(\bF)$. More generally, subspaces of $\bP^N$
correspond to linear systems of hypersurfaces of degree $d$ on $\bP^n$. Restricting to the case $d=2$, we recall the following well-known facts:
\begin{enumerate}
    \item the hyperplanes sections of $\VV_{2,2}(\bF)$ in $\mathbb{P}^5$ are the image of conics in $\mathbb{P}^2$,
    \item the hyperplanes sections of $\VV_{3,2}(\bF)$ in $\mathbb{P}^{9}$ are the image of quadrics in $\mathbb{P}^3$, and
    \item the linear section of $\VV_{3,2}(\bF)$ with a subspace of co-dimension two in $\mathbb{P}^{9}$ is the image of the base of a pencil of quadrics in $\mathbb{P}^3$.
\end{enumerate}

\section{Identifiable Waring subspaces with respect to the quadric Veronesean}\label{sec:Waring}

The rest of the paper is devoted to the study of properties, and to provide examples, of identifiable Waring subspaces with respect to the quadric Veronesean $\VV_{n,2}$.

\smallskip

We start by giving some trivial examples of identifiable Waring subspaces. First of all, since the quadric Veronesean is the intersection of quadrics and does not contain any lines, each two distinct points $P_1$ and $P_2$ on the quadric Veronesean $\VV_{n,2}$ determine a line $\langle P_1,P_2\rangle$ which meets $\VV_{n,2}$ in $P_1$ and $P_2$ only, i.e. $\langle P_1,P_2\rangle$ is an identifiable Waring line with respect to $\VV_{n,2}$.

\begin{remark}\label{rk:smaller}
If $\mathcal{S}=\langle P_1,\ldots,P_{k} \rangle$ is an identifiable Waring subspace, then 
$$\langle P_{i_1},\ldots,P_{i_{\ell+1}} \rangle$$ 
is an identifiable Waring subspace of dimension $\ell$ for any choice of $i_1,\ldots,i_{\ell+1}$ distinct elements in  $\{1,\ldots,k\}$.
\end{remark}

\smallskip

A general construction of an identifiable Waring subspace is given by the following result, which is well-known.
In the proof, and in the remainder of the paper, the vectors $\mathbf{e}_0,\ldots,\mathbf{e}_n$ denote the vectors of a basis of $V\cong\bF^{n+1}$.
\begin{theorem}
The subspace $\mathcal{S}$ spanned by the images of a frame in $\PP^n$ is an $(n+1)$-dimensional identifiable Waring subspace with respect to $\VV_{n,2}$.
\end{theorem}
\begin{proof}
Without loss of generality we may assume that
\[ \mathcal{S}=\left\langle \mathbf{e}_0^{\otimes2},\ldots, \mathbf{e}_n^{\otimes2}, (\mathbf{e}_0+\ldots+\mathbf{e}_{n})^{\otimes2} \right\rangle. \]
Each point of $\mathcal{S}$ correspond to a matrix of the form
$$
M_{a_0,\ldots,a_{n+1}}:=a_{n+1}I_{n+1}+diag(a_0,\ldots,a_n),
$$
with $a_0,\ldots,a_{n+1} \in \bF_q$.
Clearly, $\mathcal{S}$ is an identifiable Waring subspace if and only if the only case in which the matrix 
$M_{a_0,\ldots,a_{n+1}}$ has rank one is when $n+1$ of the elements $a_0,\ldots,a_{n+1}$ are zero and the remaining element is nonzero. Clearly $M_{a_0,\ldots,a_{n},0}$
has rank one if and only if $n$ elements in $\{a_0,\ldots,a_{n}\}$ are zero and the remaining one is non-zero.
Now, if $a_{n+1}\neq0$, then we may assume that $a_{n+1}=1$ and in this case
$M_{a_0,\ldots,a_{n},1}$ has rank one if and only if all its rows coincide, which happens if and only if $a_1=a_2=\ldots=a_n=0$.
Hence, the assertion is proved.
\end{proof}

As a consequence, we have the following.

\begin{corollary}
If $\mathcal{S}$ is generated by the images under $\nu_{2}$ of $\ell+1$ points contained in a frame of $\PP^n$, then $\mathcal{S}$ is an $\ell$-dimensional identifiable Waring subspace with respect to $\VV_{n,2}$.
\end{corollary}

The above results give two types of constructions for an $\ell$-dimensional identifiable Waring subspace $\mathcal{S}$ with respect to $\VV_{n,2}$.
\begin{enumerate}
  \item[(S1)] The image under $\nu_{2}$ of a frame of an $\ell$-dimensional subspace in $\PP^n$.
  \item[(S2)] The image under $\nu_2$ of $\ell+1$ points which are in general position in $\PP^n$.
\end{enumerate}
In both cases
\[ \mathcal{S}=\langle \nu_{2}(P_1),\ldots,\nu_{2}(P_{\ell+1})\rangle \]
is an identifiable Waring subspace with respect to $\VV_{n,2}$.

\section{Identifiable Waring subspaces with respect to $\VV_{2,2}(\bF_q)$}\label{sec:2,2}

Identifiable Waring planes and Waring solids of $\PP^5_q$ with respect to $\VV_{2,2}$ are easy to construct, see constructions (S1) and (S2). The next result classifies the identifiable Waring subspaces with respect to $\VV_{2,2}$ of any dimension in $\PP^5_q$.

\begin{theorem}\label{thm:V_2,2}
$(i)$ The identifiable Waring lines of $\bP_q^5$ with respect to $\VV_{2,2}$ are the span of the images under $\nu_2$ of any two distinct points in $\PP^2_q$.\\
$(ii)$ If $q>2$, then the identifiable Waring planes are those obtained by contruction (S2) with $\ell=2$, while for $q=2$, the identifiable Waring planes are the span of the images under $\nu_2$ of any three points in $\PP_q^2$.\\
$(iii)$ For $q>2$, the identifiable Waring solids are those obtained by contruction (S1) with $\ell=2$, whereas for $q=2$, the identifiable Waring solids are obtained as the span of the images of four non-collinear points in $\PP_q^2$.\\
$(iv)$ The identifiable Waring hyperplanes are the images under $\nu_2$ of irreducible conics for $q=4$, and two distinct lines for $q=2$.
\end{theorem}

\begin{proof}
$(i)$ Identifiable Waring lines. As $\VV_{2,2}$ is a the intersection of quadrics of $\PP^5$ and does not contain any lines, two distinct points of $\VV_{2,2}$ identify a $2$-secant line to $\VV_{2,2}$. Therefore, the span of the images under $\nu_2$ of any two distinct points of $\PP^2$ is an identifiable Waring line.

$(ii)$ Identifiable Waring planes. Since the images under $\nu_2$ of three collinear points in $\PP_q^2$ determine a $(q+1)$-secant plane to $\VV_{2,2}(\bF_q)$, if $q>2$, identifiable Waring planes can only be obtained by contruction (S2).

$(iii)$ Identifiable Waring solids. If $q>2$, an identifiable Waring subspace of dimension $3$ of $\PP_q^5$ cannot be obtained as the span of the images under $\nu_2$ of four points in $\PP_q^2$ of which (at least) three are collinear. Indeed, consider $P_1, P_2, P_3, P_4$ points in $\PP_q^2$. Clearly, if they are collinear, then $\dim \langle \nu_2(P_1), \nu_2(P_2), \nu_2(P_3), \nu_2(P_4) \rangle =2$. Assume now that $P_1, P_2, P_3$ are collinear and $P_4 \notin \langle P_1, P_2\rangle$. Then 
$$|\langle \nu_2(P_1), \nu_2(P_2), \nu_2(P_3), \nu_2(P_4) \rangle \cap\VV_{2,2}(\bF_q)|\geq q+2,$$ 
which is always greater than $4$.
For $q=2$ the images under $\nu_2$ of any $4$ points in $\PP_q^2$ which are not collinear determine an identifiable Waring solid.

$(iv)$ Identifiable Waring hyperplanes.
The hyperplanes sections of $\PP^5$ are in one to one correspondence with the conics of $\PP^2$.
The conics of $\PP_q^2$ have either $1$, $q+1$ (irreducible conics or double line) or $2q+1$ (two distinct lines) points.
Therefore, identifiable Waring hyperplanes of $\PP_q^5$ are those arising from irreducible conics if $q=4$ and those from two distinct lines for $q=2$.
\end{proof}

The above result allows us to completely determine the identifiable Waring polynomial of $\mathbb{V}_{2,2}$ over $\bF_q$.

\begin{corollary}
The identifiable Waring polynomial of the quadric Veronesean $\VV_{2,2}$ over $\bF_q$, with $q\neq 2,4$, is 
\begin{eqnarray}\label{eqn:V_2,2}
 \mathcal{IW}_{\VV_{2,2}}(\bF_{q})=1+X+X^2+X^3.
\end{eqnarray}
For $q=2$ and $q=4$ the identifiable Waring polynomials are
$$
\mathcal{IW}_{\VV_{2,2}}(\bF_{2})=1+X+2X^2+2X^3+X^4, \mbox{ and } 
\mathcal{IW}_{\VV_{2,2}}(\bF_{4})=1+X+X^2+X^3+X^4.
$$

\end{corollary}

\begin{proof}
The term $1+X$ (for all $q$) follows from the well-known fact that $\Aut(\VV_{2,2}(\bF_q))$ acts two-transitively on the points of $\VV_{2,2}(\bF_q)$. The term $X^2+X^3$ in the cases $q\neq 2$ follows from Theorem \ref{thm:V_2,2} and the well-known fact that $\PGL(3,q)$ acts transitively on triangles and on frames of $\bP_q^2$.
Assume $q=2$.
By the previous theorem the identifiable Waring planes are the span of the images under $\nu_2$ of any three points in $\PP_q^2$, so there will be two equivalence classes depending on whether these three points of $\PP_q^2$ are collinear or not. 
Furthermore, by the previous result the identifiable Waring solids are obtained as the span of the images of four non-collinear points in $\PP_q^2$, and hence the two equivalence classes can be identified according to whether the four points are in general position or not.
\end{proof}

Note that when $q\neq 2,4$ there do not exist identifiable Waring hyperplanes with respect to $\VV_{2,2}(\bF_q)$.

\begin{remark}
The above results can also be recovered from the following (much more general) classification results. All lines with respect to $\VV_{2,2}(\bF_q)$ have been classified in \cite{LP} for both $q$ even and $q$ odd. Planes meeting $\VV_{2,2}(\bF_q)$ have been classified in \cite{LPS} for $q$ odd, and in \cite{AlLaPo2022} for $q$ even. The classification of solids with respect to $\VV_{2,2}(\bF_q)$ follows from the classification of lines for $q$ odd, while for $q$ even this classification was obtained in \cite{AlLa202*}. 
\end{remark}
\begin{remark}
In terms of linear systems of conics the identifiable Waring polynomial \eqref{eqn:V_2,2} corresponds to the fact that, up to projective equivalence, there is a unique web of conics in $\bP_q^2$ with base two points, there is a unique net of conics with base three non-collinear points, and there is a unique pencil of conics with base a frame.
\end{remark}

\section{Identifiable Waring subspaces with respect to $\VV_{3,2}(\bF_q)$}\label{sec:3,2}

The aforementioned constructions (S1) and (S2) give examples of identifiable Waring subspaces of $\bP_q^9$ of dimension one, two, three and four with respect to $\VV_{2,2}(\bF_q)$.
In this section we investigate the existence of identifiable Waring subspaces of dimension at least $5$. 

\subsection{Identifiable Waring subspaces of codimension $\geq 3$}
We construct examples of identifiable Waring subspaces of dimension six in $\PP_q^9$ for all $q$ (and therefore also of dimension less than or equal to $6$ by Remark \ref{rk:smaller}).
\begin{theorem}\label{6-Waring}
The subspace
\[ \mathcal{S}=\langle \mathbf{e}_1^{\otimes2},\mathbf{e}_2^{\otimes2},\mathbf{e}_3^{\otimes2},\mathbf{e}_4^{\otimes2},  \mathbf{e}^{\otimes2},(\mathbf{e}_1-\mathbf{e}_2-\mathbf{e}_3)^{\otimes2},(-\mathbf{e}_1+\mathbf{e}_2-\mathbf{e}_4)^{\otimes2} \rangle, \]
where $\mathbf{e}=\mathbf{e}_1+\mathbf{e}_2+\mathbf{e}_3+\mathbf{e}_4$, is a $6$-dimensional identifiable Waring subspace of $\PP_q^9$ with respect to $\VV_{3,2}(\bF_q)$ when $q$ is odd.
\end{theorem}
\begin{proof}
The points of $\mathcal{S}$ can be represented by the matrices
\begin{eqnarray}\label{eqn:M1}
 M_\alpha:=\left(
\begin{array}{cccc}
 a+e+f-g & e-f+g & e-f & e-g \\
 e-f+g & b+e+f-g & e+f & e+g \\
 e-f & e+f & c+e+f & e \\
 e-g & e+g & e & d+e-g \\
\end{array}
\right),
\end{eqnarray}
with $\alpha=(a,b,c,d,e,f,g) \in \bF_q^7\setminus \{0\}$. One easily verifies that $\mathcal S$ is a 6-dimensional subspace of $\PP^9$. In order to prove that $\mathcal S$ is a Waring subspace, it suffices to show that the only matrices of rank one of the form (\ref{eqn:M1}) are obtained by setting all but one of the parameters $a,b,c,d,e,f,g$ equal to zero.
For simplicity of notation we denote the $2\times 2$-minor of $M_\alpha$ corresponding to the rows $(i,j)$ and columns $(k,l)$ by $M_{(i,j)(k,l)}$.

Suppose $M_\alpha$ has rank one for some choice of $\alpha=(a,b,c,d,e,f,g) \in \bF_q^7\setminus \{0\}$. Then 
$$M_{(1,2)(3,4)}=-2e(f-g)=0$$ 
which implies that $e=0$ or $f=g$, since $q$ is odd.

\bigskip

$(i)$ First suppose $e=0$. Then 
$$M_{(3,4)(2,3)}=g(c+f)=0$$ 
implies $g=0$ or $c+f=0$.

If $g=0$ then $M_{(2,3)(1,2)}=0$ implies $bf=0$. If $f\neq 0$ and $b=0$, then $M_{(2,3)(2,3)}=0$ implies that $c=0$. Similarly by considering $M_{(1,2)(1,2)}$ and $M_{(3,4)(3,4)}$ it also follows that $a=0$ and $d=0$. So that in this case $a=b=c=d=e=g=0$ and $f\neq 0$. If $e=g=0$ and $f=0$, then the matrix $M_\alpha$ reduces to a diagonal matrix with $a,b,c,d$ on the diagonal. This matrix has rank one if and only if all but one of the diagonal elements are equal to zero. So, all but one of the parameters $a,b,c,d,e,f,g$ must be zero.

If $g\neq 0$ then $c+f$ must be zero. The condition $M_{(2,3)(2,3)}=0$ implies that $f=0$, and therefore also $c=0$. By considering $M_{(1,2)(2,4)}$ and $M_{(1,2)(1,2)}$ it also follows that $b=0$ and $a=0$. Finally, $M_{(1,4)(2,4)}=0$ implies $d=0$. This concludes the case $e=0$.

\bigskip

$(ii)$ Now suppose $e\neq 0$. Then $f$ and $g$ must be equal. The condition $M_{(1,4)(2,3)}=0$ implies that $f=0$, and therefore also $g=0$. Then $M_{(1,2)(2,3)}=0$ implies that $b=0$ and $M_{(1,2)(1,2)}=0$ implies that $a=0$. Similar arguments show that also $c$ and $d$ must be zero. We may conclude that in the case $e\neq 0$ it follows that the only matrices of rank one of the form (\ref{eqn:M1}) are obtained by setting $a=b=c=d=f=g=0$, as was required.
\end{proof}

\begin{remark}
Since the proof only relies on the fact that $2\ne0$, the above example not only works over finite fields, but works over any field of characteristic either zero or greater than two.
\end{remark}

\begin{remark}
Computational results show that the above example works for $q=2$, but not for $q\in \{4,8\}$.
Furthermore, if we consider the projective subspace of $\PP^9$ defined by the span
\[\langle W,(\mathbf{e}_2-\mathbf{e}_3+\mathbf{e}_4)^{\otimes2}\rangle,\]
where $W$ is as in Theorem \ref{6-Waring}, then it is an identifiable Waring subspace for $q=3$, but not for $q \in \{5,7,9\}$.
\end{remark}

The next theorems provide examples of identifiable Waring subspaces with respect to $\VV_{3,2}(\bF_q)$ for all $q\geq 4$.

\begin{theorem}\label{6-Waring}
Let $q\geq 4$. The subspace
\[ \mathcal{S}=\langle \mathbf{e}_1^{\otimes2},\mathbf{e}_2^{\otimes2},\mathbf{e}_3^{\otimes2},\mathbf{e}_4^{\otimes2},\mathbf{e}^{\otimes2},  (\omega\mathbf{e}_1+\mathbf{e}_2+\omega\mathbf{e}_3)^{\otimes2},(\sqrt{\omega}\mathbf{e}_1+\sqrt{\omega}\mathbf{e}_2+\sqrt{\omega^3}\mathbf{e}_4)^{\otimes2}  \rangle, \]
where $\mathbf{e}=\mathbf{e}_1+\mathbf{e}_2+\mathbf{e}_3+\mathbf{e}_4$, 
is a $6$-dimensional identifiable Waring subspace of $\PP_q^9$ with respect to $\VV_{3,2}(\bF_q)$, for every $\omega \in\bF_q\setminus\{0,1,2\}$ which is a square in $\bF_q$.
\end{theorem}
\begin{proof}
The points of $\mathcal{S}$ can be represented by the matrices
\begin{equation}\label{eqn:M2}
 M_{\alpha}:=\left(
\begin{array}{cccc}
 f \omega^2+g \omega+a+e & e+f \omega+g \omega & f \omega^2+e & g \omega^2+e \\
 e+f \omega+g \omega & b+e+f+g \omega & e+f \omega & g \omega^2+e \\
 f \omega^2+e & e+f \omega & f \omega^2+c+e & e \\
 g \omega^2+e & g \omega^2+e & e & g \omega^3+d+e \\
\end{array}
\right),
\end{equation}
with $\alpha=(a,b,c,d,e,f,g) \in \bF_q^7\setminus \{0\}$. One easily verifies that $\mathcal S$ is a 6-dimensional subspace of $\PP_q^9$. As in the proof of Theorem \ref{6-Waring}, in order to prove that $\mathcal S$ is a Waring subspace, it suffices to show that the only matrices of rank one of the form \eqref{eqn:M2} are obtained by setting all but one of the parameters $a,b,c,d,e,f,g$ equal to zero.
Again, we denote the $2\times 2$-minor of $M_\alpha$ corresponding to the rows $(i,j)$ and columns $(k,l)$ by $M_{(i,j)(k,l)}$.
Assume that $M_{\alpha}$ has rank one for some $\alpha=(a,b,c,d,e,f,g)\in \mathbb{F}_q^7\setminus\{0\}$. Then 
$$M_{(3,4)(1,2)}=f(\omega-1)\omega(e+g\omega^2)=0,$$ 
which implies either $f=0$ or $e=-g\omega^2$.

\bigskip

(1) First suppose that $f=0$. Then $M_{(1,3)(1,2)}=ae=0$, that is either $a=0$ or $e=0$.

(1.1) Let $a=0$. Then $M_{(1,2)(2,3)}=be=0$. If $b=0$, then $M_{(1,3)(1,4)}=eg\omega(1-\omega)=0$, which implies either $e=0$ or $g=0$. If $e=0$, then by $M_{(1,3)(1,3)}$, $M_{(1,4),(1,4)}$ and by $M_{(3,4)(3,4)}$ we have $cg=cd=dg=0$.
Whereas, if $a=b=g=0$,
\[ M_{\alpha}=\left( \begin{array}{cccc}
e & e & e & e \\
e & e & e & e \\
e & e & c+e & e \\
e & e & e & d+e
\end{array} \right). \]
Thus, in both cases all but one of the parameters $a,b,c,d,e,f,g$ must be zero.
Assume now $a=e=0$ and $b\ne 0$. By $M_{(1,2),(1,2)}$, $M_{(2,4)(2,4)}$ and by $M_{(2,3)(2,3)}$ we get $c=d=g=0$, that is $a=c=d=e=f=0$ and $b\ne 0$.

(1.2) Suppose that $e=0$ and $a\ne 0$. Then $M_{(1,2)(1,4)}=agw^2=0$, which implies $g=0$. So, $M_{\alpha}$ reduces to a diagonal matrix with $a,b,c,d$ on the diagonal. Hence, as before, $M_\alpha$ has rank one if and only if all but one of the parameters $a,b,c,d,e,f,g$ are zero.

\bigskip

(2) Suppose that $e=-g\omega^2$ and $f\ne 0$. Then, we have $M_{(1,3)(3,4)}=-g\omega^4(f-g)=0$, and hence either $g=0$ or $f=g$.
If $f=g$, then 
$$M_{(1,3)(2,4)}=g^2(\omega-2)\omega^3=0,$$ 
and so $g=f=0$, which leads to a contradiction.
If $g=0$, then $e=0$. By $M_{(1,2)(1,3)}$, $M_{(1,2)(1,2)}$, $M_{(1,3)(1,3)}$ and $M_{(1,4)(1,4)}$ we have $a=b=c=d=0$. Again we have that all but one of the parameters $a,b,c,d,e,f,g$ are zero.
\end{proof}

\begin{theorem}\label{thm:codim_three_P9}
Let $\omega$ be an element in $\bF_q\setminus\{0,1,-1\}$.
The subspace
\[ \mathcal{S}=\langle \mathbf{e}_1^{\otimes2},\mathbf{e}_2^{\otimes2},\mathbf{e}_3^{\otimes2},\mathbf{e}_4^{\otimes2},  (\mathbf{e}_2+\mathbf{e}_3+\mathbf{e}_4)^{\otimes2},\mathbf{v},\mathbf{w} \rangle, \]
with $\mathbf{v}=(\mathbf{e}_1+\omega \mathbf{e}_2+\omega \mathbf{e}_3+\omega^2\mathbf{e}_4)^{\otimes2}$ and $\mathbf{w}=(\mathbf{e}_1+\omega \mathbf{e}_2+ \mathbf{e}_3+\omega \mathbf{e}_4)^{\otimes2}$,
is a $6$-dimensional identifiable Waring subspace of $\PP_q^9$ with respect to $\VV_{3,2}(\bF_q)$.
\end{theorem}
\begin{proof}
The nonzero vectors of $\mathcal{S}$ can be represented by the matrices
{\footnotesize{
\begin{equation}\label{eqn:M_codim2}
 M_\alpha:=\left(
\begin{array}{cccc}
 a+f+g & \omega f+\omega g & \omega f+g & \omega^2 f +\omega g  \\
 \omega f+\omega g & b+e+\omega^2 f+\omega^2 g & e+\omega^2 f+\omega g & e+\omega^3 f +\omega^2 g \\
 \omega f+g & e+\omega^2 f +\omega g & c+e+\omega^2 f+g & e+\omega^3 f +\omega g \\
 \omega^2 f+\omega g  & e+\omega^3 f+\omega^2 g & e+\omega^3 f+\omega g & d+e+\omega^4 f +\omega^2 g \\
\end{array}
\right),
\end{equation} }}
with $\alpha=(a,b,c,d,e,f,g) \in \bF_q^7\setminus \{0\}$. A straightforward computation shows that $\mathcal S$ is a 6-dimensional subspace of $\PP^9$. In order to prove that $\mathcal S$ is a Waring subspace, it suffices to show that the only matrices of rank one of the form (\ref{eqn:M_codim2}) are obtained by setting all but one of the parameters $a,b,c,d,e,f,g$ equal to zero.
Since
$$M_{(2,3)(1,4)}=g (-1 + \omega) (e + f (-1 + \omega) \omega^2)=0,$$
we have to analyze the case $g=0$ and the case $g \ne 0$ and $e=-f\omega^2(\omega-1)$.

\bigskip

$(i)$ Suppose $g=0$, then $M_{(2,3)(1,2)}=-bf\omega=0$, that is either $b=0$ or $f=0$.
If $b=0$, then $M_{(2,4)(1,2)}=-\omega(\omega-1)ef=0$.
If $e=0$, by evaluating $M_{(1,2)(1,2)}$, $M_{(1,3)(2,3)}$, $M_{(1,3)(1,3)}$, $M_{(2,3)(1,4)}$, $M_{(1,4)(1,4)}$ and $M_{(3,4)(3,4)}$ we get $af=cf=ac=df=ad=cd=0$.
If $f=0$ and $e \ne 0$, by $M_{(1,2)(1,2)}=M_{(1,3)(1,3)}=M_{(1,4)(1,4)}=0$ we get $e=c=d=0$.
Now, if $b \ne 0$ and $f=0$, from $M_{(1,2)(1,2)}=M_{(2,3)(2,3)}=M_{(2,4)(2,3)}=M_{(2,4)(2,4)}=0$ we get $a=c=d=e=0$.
In these cases, all but one of the parameters $a,b,c,d,e,f,g$ equal to zero.

\bigskip

$(ii)$ Assume $g \ne 0$ and $e=-f\omega^2(\omega-1)$. The condition that 
$$M_{(3,4)(1,2)}=f (-1 + \omega)^2 \omega^2 (g + f w)=0,$$ 
implies that either $f=0$ or $g=-\omega f$. If $f=0$, then $e=0$ and, since $g\ne 0$, we get $a=b=c=d=0$, since $M_{(1,3)(1,2)}=M_{(2,3)(1,2)}=M_{(2,3)(1,3)}=M_{(2,3)(1,4)}=0$.
Again, all but one of the parameters $a,b,c,d,e,f,g$ equal to zero, as required.
Suppose now $g=-\omega f$ and $f \ne 0$.
By $M_{(2,3)(1,2)}=(\omega-1)^2 \omega^3f^2=0$, we get $f=0$, a contradiction.
\end{proof}

We now show that the above examples from Theorem \ref{6-Waring} and Theorem \ref{thm:codim_three_P9} are not equivalent.
\begin{proposition}
Let $q\geq 4$, $\omega \in\bF_q\setminus\{0,1,2\}$ which is a square in $\bF_q$ and $\omega' \in \bF_q\setminus\{0,1,-1\}$. The subspaces
\[ \mathcal{S}=\langle \mathbf{e}_1^{\otimes2},\mathbf{e}_2^{\otimes2},\mathbf{e}_3^{\otimes2},\mathbf{e}_4^{\otimes2},\mathbf{e}^{\otimes2},  (\omega\mathbf{e}_1+\mathbf{e}_2+\omega\mathbf{e}_3)^{\otimes2},(\sqrt{\omega}\mathbf{e}_1+\sqrt{\omega}\mathbf{e}_2+\sqrt{\omega^3}\mathbf{e}_4)^{\otimes2}  \rangle, \]
and
\[ \mathcal{S}'=\langle \mathbf{e}_1^{\otimes2},\mathbf{e}_2^{\otimes2},\mathbf{e}_3^{\otimes2},\mathbf{e}_4^{\otimes2},  (\mathbf{e}_2+\mathbf{e}_3+\mathbf{e}_4)^{\otimes2},\mathbf{v},\mathbf{w} \rangle, \]
with $\mathbf{v}=(\mathbf{e}_1+\omega' \mathbf{e}_2+\omega' \mathbf{e}_3+\omega'^2\mathbf{e}_4)^{\otimes2}$ and $\mathbf{w}=(\mathbf{e}_1+\omega' \mathbf{e}_2+ \mathbf{e}_3+\omega' \mathbf{e}_4)^{\otimes2}$,
are inequivalent.
\end{proposition}
\begin{proof}
Recall that $\mathcal{S}$ and $\mathcal{S}'$ are equivalent if and only if the set of points $S$ defined by the vectors in $$\{\mathbf{e}_1,\mathbf{e}_2,\mathbf{e}_3,\mathbf{e}_4,\mathbf{e},  \omega\mathbf{e}_1+\mathbf{e}_2+\omega\mathbf{e}_3,\sqrt{\omega}\mathbf{e}_1+\sqrt{\omega}\mathbf{e}_2+\sqrt{\omega^3}\mathbf{e}_4\}$$ 
and the set of points $S'$ defined by the vectors in 
$$\{\mathbf{e}_1,\mathbf{e}_2,\mathbf{e}_3,\mathbf{e}_4,  \mathbf{e}_2+\mathbf{e}_3+\mathbf{e}_4,\mathbf{e}_1+\omega \mathbf{e}_2+\omega \mathbf{e}_3+\omega^2\mathbf{e}_4,\mathbf{e}_1+\omega \mathbf{e}_2+ \mathbf{e}_3+\omega \mathbf{e}_4\},$$ 
are $\PGL(4,q)$-equivalent, i.e. $\phi(S)=S'$, for some $\phi \in \mathrm{PGL}(4,q)$.
Now, observe that the number of planes of $\mathbb{P}^3_q$ meeting $S$ in exactly $4$ points is $4$, namely the planes $\pi_1:X_3=0$, $\pi_1:X_2=0$, $\pi_3:X_0-X_2=0$ and $\pi_4:X_0-X_1=0$. On the other hand, each of the planes $\pi'_1:X_0=0$, $\pi'_2:\omega' X_2-X_3=0$, $\pi'_3:X_1-X_3=0$, $\pi'_4:X_2-X_3=0$ and $\pi'_5:\omega' X_0-X_1=0$, meets
$S'$ in exactly $4$ points. Therefore, such a $\phi \in \mathrm{PGL}(4,q)$ cannot exist and $\mathcal{S}$ and $\mathcal{S}'$ are inequivalent.
\end{proof}
In particular, the above result shows that the coefficient 
$\eta_6(\mathbb{V}_{3,2}(\F_q))$ in the identifiable Waring polynomial of ${\mathcal{IW}}_{\mathbb{V}_{3,2}}(\F_q))$
is $\geq 2$ when $q \geq 4$.

\subsection{$7$-dimensional identifiable Waring subspaces}

In this section we give a construction of $7$-dimensional identifiable Waring subspaces in $\PP_q^9$.
The proof relies on the following lemma concerning the existence of $\bF_q$-rational points on an algebraic curve satisfying certain conditions. 

For future reference we define the set 
$$\mathcal B_*$$
as the set of pairs $(q,\omega)$ satisfying one of the following conditions: $q\in \{4,5,7,8,9\}$ and $\omega$ a primitive element in $\bF_q$, $q=11$ and $\omega \in \{7,8\}$, or $q=13$ and $\omega\in \{2,7\}$.

\begin{lemma}\label{lemma:cubic}
Let $q$ be a prime power and let $\omega \in \mathbb{F}_q\setminus\{0,1,-1\}$.
The curve $\mathcal{C}$ with affine equation
\[ (\omega-1)^2X^2Y+\omega Y + (\omega-1)^2\omega XY^2+\omega^2 X+ (\omega-1)^2(\omega+1)XY=0  \]
is absolutely irreducible.
Also, it has no $\mathbb{F}_q$-rational point $(x,y)$ such that $x\ne-y$, $x\ne \frac{\omega}{\omega-1}$, $x\ne \frac{1}{\omega -1}$ and $1+x+\omega y \ne 0$ if and only if $(q,\omega)\in \mathcal B_*$.
\end{lemma}
\begin{proof}
We start by proving that the planar cubic $\cC$ defined by the following polynomial
\[ P(X,Y)=(\omega-1)^2X^2Y+\omega Y + (\omega-1)^2\omega XY^2+\omega^2 X+ (\omega-1)^2(\omega+1)XY \]
is absolutely irreducible.
Suppose, by way of contradiction, that $\cC$ is reducible over some field extension of $\bF_q$, then one of its component is a line $\ell$ having as point at infinity one of those of $\C$.
The points at infinity of $\C$ are $(1:0:0)$, $(0:1:0)$ and $(-\omega:1:0)$.

If $\ell$ has $(1:0:0)$ as its point at infinity, then $\ell$ has equation $Y=k$, for some $k \in \overline{\mathbb{F}}_q$, the algebraic closure of $\bF_q$.
In particular, the polynomial
\[ P(X,k)=k(\omega-1)^2 X^2+k\omega+k^2(\omega-1)^2\omega X+\omega^2 X+k(\omega-1)^2(\omega+1) X \]
should be the zero polynomial. This is a contradiction since $\omega \ne 0$.

If $\ell$ has $(0:1:0)$ as its point at infinity, then $\ell$  has equation $X=k$, for some $k \in \overline{\mathbb{F}}_q$.
In particular, the polynomial
\[ P(k,Y)=(\omega-1)^2k^2Y+\omega Y+\omega(\omega-1)^2 k Y^2+\omega^2 k+(\omega-1)^2(\omega+1)kY \]
should be the zero polynomial. This is a contradiction since $\omega \ne 0$.

If $\ell$ has $(-\omega:1:0)$ as its point at infinity, then $\ell$ has equation $X=-\omega Y+ k$, for some $k \in \overline{\mathbb{F}}_q$.
In particular, the polynomial
\[ P(-\omega Y +k,Y)= [-2(\omega-1)^2\omega+(\omega-1)^2\omega k-\omega (\omega-1)^2(\omega+1)]Y^2+\]
\[ [(\omega-1)^2k^2+\omega-\omega^3+(\omega-1)^2(\omega+1)k]Y+\omega^2 k \]
should be the zero polynomial.
The constant term implies that $k=0$, and the coefficient of the term of degree one implies $\omega^3-\omega=0$. This is again a contradiction, since $\omega \notin \{0,\pm 1\}$.

This shows that the cubic curve $\cC$ is absolutely irreducible.
Note that the lines 
$$
\ell_1 \colon X+Y=0, \quad \ell_2 \colon X= \frac{\omega}{\omega-1}, \quad \ell_3 \colon X= \frac{1}{\omega -1}, \mbox{ and } \ell_4 \colon 1+X+\omega Y=0
$$
can share with the curve $\mathcal{C}$ at most three points.
Therefore, denoting by $\mathcal{C}(\mathbb{F}_q)$ the set of $\mathbb{F}_q$-rational points of $\mathcal{C}$, if $|\mathcal{C}(\mathbb{F}_q)|\geq 14$, there exists at least one $\mathbb{F}_q$-rational affine point not lying on the lines $\ell_1$, $\ell_2$, $\ell_3$ and $\ell_4$. As $\mathcal{C}$ is a cubic curve, its genus is either $0$ or $1$, and by the Hasse-Weil bound we have
\[ |\mathcal{C}(\mathbb{F}_q)|\geq q+1-2\sqrt{q}. \]
Therefore, when $q\geq 23$, we have $|\mathcal{C}(\mathbb{F}_q)|\geq14$, and hence there exists an affine $\mathbb{F}_q$-rational point of $\mathcal{C}$, with affine coordinates $(x,y)$, satisfying $x\ne-y$, $x\ne \frac{\omega}{\omega-1}$, $x\ne \frac{1}{\omega -1}$ and $1+x+\omega y \ne 0$.
The remaining cases $q\in \{16,17,19\}$ and the determination of the values of $\omega$ for $q \in \{4,5,7,8,9,11,13\}$ are easily verified by computer.
\end{proof}

\begin{theorem}\label{thm:codim_two_P9}
Let $\omega$ be an element in $\bF_q\setminus\{0,1,-1\}$, and let $\mathcal S$ be the subspace defined in Theorem \ref{thm:codim_three_P9}.
The subspace
\[ \mathcal{S}'=\langle \mathcal{S},(\mathbf{e}_1+\mathbf{e}_3+\omega^2 \mathbf{e}_4)^{\otimes2}\rangle  \]
is a $7$-dimensional identifiable Waring subspace of $\PP_q^9$ with respect to $\VV_{3,2}(\bF_q)$ if and only if $(q,\omega)\in \mathcal B_*$. In particular, $\mathcal{S}'$ is not an identifiable Waring subspace for $q\geq 16$.
\end{theorem}

\begin{proof}
The nonzero vectors of $\mathcal{S}'$ can be represented by the matrices $M_\alpha$, with $\alpha=(a,b,c,d,e,f,g,h) \in \bF_q^8\setminus \{0\}$, where $M_\alpha$ is the matrix
{\footnotesize{
$$%\begin{equation}%\label{eqn:M}
\left(
\begin{array}{cccc}
 a+f+g+h & \omega f+\omega g & \omega f+g+h & \omega^2 f +\omega g+\omega^2 h  \\
 \omega f+\omega g & b+e+\omega^2 f+\omega^2 g & e+\omega^2 f+\omega g & e+\omega^3 f +\omega^2 g \\
 \omega f+g+h & e+\omega^2 f +\omega g & c+e+\omega^2 f+g+h & e+\omega^3 f +\omega g+\omega^2 h \\
 \omega^2 f+\omega g+ \omega^2 h  & e+\omega^3 f+\omega^2 g & e+\omega^3 f+\omega g+\omega^2 h & d+e+\omega^4 f +\omega^2 g+\omega^4 h \\
\end{array}
\right).
$$ %\end{equation} 
}}
One easily verifies that $\mathcal S'$ is a 7-dimensional subspace of $\PP^9$. In order to prove that $\mathcal S'$ is a Waring subspace, it suffices to show that the only matrices of rank one of the form $M_\alpha$ are obtained by setting all but one of the parameters $a,b,c,d,e,f,g,h$ equal to zero.

By considering $M_{(1,4)(2,3)}$ we get
\[ e(-1+(\omega-1)g)=-g(1+f(\omega-1))\omega^2(\omega -1). \]
Note that $g\ne \frac{1}{\omega-1}$, otherwise by considering $M_{(2,3)(1,4)}=0$ we get $f=\frac{-1}{\omega-1}$, and from the condition $M_{(2,4)(1,4)}=0$ we would get $\omega(\omega-1)(\omega^2-\omega)=0$, which is not possible.
Therefore, we have 
$$e=\frac{-g(1+f(\omega-1))\omega^2(\omega -1)}{-1+(\omega-1)g}.$$
Computing $M_{(2,3)(1,2)}=0$ we get
\[ b(1+g+\omega f)=-\omega^2[f (1+g (\omega-1)^2)+\omega g]. \]
Note that $1+g+\omega f\neq 0$, otherwise from $M_{(2,3)(1,2)}=0$ and from $M_{(1,3)(1,2)}=0$ we would get $f=\frac{-1}{\omega-1}$ and $a=-1$, from which it would follow that $M_{(2,4)(1,4)}=\omega^3(\omega-1)=0$, a contradiction. Hence, 
$$b=-\omega^2\frac{f (1+g (\omega-1)^2)+\omega g}{1+g+\omega f}.$$ 
From $M_{(1,2)(1,4)}=0$, we get either $f+g=0$ or 
\[ a(g+\omega-\omega g )=-g(\omega-1)[f(\omega-1)+\omega]. \]
If $g=-f$, then by $M_{(2,3)(2,3)}=0$ we get $f=0$ and one easily get $a=b=c=d=e=f=g=0$.
Also, $g\neq \frac{\omega}{\omega-1}$, otherwise we would get a contradiction as before.
Hence, we may assume that $f+g\ne 0$, $g\ne \frac{\omega}{\omega-1}$, $g\ne \frac{1}{\omega-1}$ and 
\[ a=\frac{-g(\omega-1)[f(\omega-1)+\omega]}{g+\omega-\omega g}. \]
The condition $M_{(2,3)(1,3)}=0$ implies
\[ c= \frac{(\omega-1)[f(1+g(\omega-1)^2)+\omega g]}{f+g}, \]
and the condition $M_{(1,4)(1,4)}=0$ implies either $g=-\omega-\omega f$ or
\begin{eqnarray}\label{eqn:d}
 d=-\frac{\omega^2(\omega-1)^2g[f(\omega-1)+\omega]}{-1+g(\omega-1)}.
\end{eqnarray}
If $g=-\omega-\omega f$, then $M_{(1,3)(1,3)}=(\omega-1)^2\ne 0$, again a contradiction.
Assume $d$ satisfies Equation (\ref{eqn:d}). Then, as $f+g\ne 0$, $1+g+\omega f \ne 0$, $(\omega-1)g-\omega \ne 0$ and $1+g+\omega f\ne 0$, remaining minors of $M$ are either zero or scalar multiples of
\[ (\omega-1)^2g^2f+\omega f+(\omega-1)^2\omega gf^2+\omega^2 g+ (\omega-1)^2(\omega+1) gf. \]
It follows that the subspace $S'$ is a Waring subspace if and only if the affine cubic curve obtained by considering the above expression as a polynomial in the indeterminates $f$ and $g$ has no $\bF_q$-rational points 
such that $g\ne-f$, $g\ne \frac{\omega}{\omega-1}$ and $g\ne \frac{1}{\omega -1}$.
Applying Lemma \ref{lemma:cubic} concludes the proof.
\end{proof}

The identifiable Waring subspaces of codimension two in $\PP_q^9$ correspond to pencils of quadrics in $\bP_q^3$, with base consisting of eight points whose image under the Veronese map spans a seven-dimensional subspace of $\PP_q^9$. As far as we know there is no complete classification of pencils of quadrics in $\bP_q^3$. In what follows we will rely on the (partial) classification results from \cite{BruenHir}.
We will also make use of the following lemma.

\begin{lemma}\label{lem:int2cones}
Let $C_1$ and $C_2$ be two irreducible quadratic cones in $\PP_q^3$, with $q\geq 53$. Then $C_1(\bF_q)$ does not meet $C_2(\bF_q)$ in eight points.
\end{lemma}
\begin{proof}
Without loss of generality, we may assume that $C_1$ has equation $X_1X_3=X_0^2$ and $C_2$ has equation $f(X_0,X_1,X_2,X_3)=0$ with 
$$f(X_0,X_1,X_2,X_3)=\sum_{i\leq j} f_{ij} X_iX_j\in \mathbb{F}_{q}[X_0,X_1,X_2,X_3].$$
Observe that $C_1$ meets the plane at infinity $\pi_{\infty}\colon X_3=0$ in the line $\ell_{\infty}\colon X_0=X_3=0$ and so either $\ell_{\infty}$ is contained in $C_2$ or meets $C_2$ in at most two $\bF_q$-rational points.
We may assume that $|\ell_{\infty} \cap C_2(\bF_q)|\leq 2$, otherwise $|C_1(\bF_q)\cap C_2(\bF_q)|\geq q+1>8$.
An affine point $(x,y,z)$ is in $C_1 \cap C_2$ if and only if its coordinates satisfy the following system
\begin{equation}\label{eq:coneint} \left\{ \begin{array}{ll} Y=X^2,\\ g(X,Y,Z)=0, \end{array}\right. \end{equation}
where $g(X,Y,Z)=f(X,Y,Z,1)$.
Let $h(X,Z)=g(X,X^2,Z)$, i.e.
\[ h(X,Z)= (f_{00}+f_{13})X^2+f_{11}X^4+f_{22}Z^2+f_{01}X^3+f_{02}XZ+f_{12}X^2Z+f_{03}X+f_{23}Z+f_{33}, \]
and observe that $h(X,Z)$ has degree at most $2$ in $Z$ and at most $4$ in $X$. 
The number of solutions in $\mathbb{F}_q^3$ of System \eqref{eq:coneint} coincides with the number of $\mathbb{F}_q$-rational affine points of the planar curve $\mathcal{C}$ defined by $h(X,Z)=0$.
Note that the number of points at infinity of $\mathcal{C}$ can be at most two (since $Y=X^2$ meets $\pi_{\infty}$ in $\ell_{\infty}$ which has at most two points in common with $C_2(\F_q)$). Hence if we prove that $\mathcal{C}$ has more than ten $\mathbb{F}_q$-rational points, then the size of $C_1(\bF_q)\cap C_2(\bF_q)$ is larger than eight.

\bigskip

First assume that $\mathcal{C}$ is absolutely irreducible. By applying the Hasse-Weil bound one gets that the number of $\mathbb{F}_q$-rational points of $\mathcal{C}$ is at least $q+1-6\sqrt{q}$ which is larger than $10$ because of the assumption on $q$.

\bigskip

Now, suppose that $\mathcal{C}$ is not absolutely irreducible. 

\bigskip

If $h(X,Z)$ splits over $\bF_q$ then $h(X,Z)$ has either a linear factor over $\bF_q$, an absolutely irreducible factor of degree two over $\bF_q$, or it is the product of $4$ linear factors over $\bF_{q^2}$ (two pairs of conjugates factors). This implies that the size of $\C(\bF_q)$ is either at least $q+1$, which is greater than $9$, or at most two. Together with the at most two points at infinity, the latter case implies that $C_1(\bF_q) \cap C_2(\bF_q)$ has size most $4$.

\bigskip

Finally, suppose that $h(X,Z)$ splits over some extension of $\bF_q$ but not over $\bF_q$.

\bigskip

If $f_{22}=0$ and $(f_{02},f_{12},f_{23})\ne (0,0,0)$, then $h(X,Z)$ is linear in $Z$, contradicting our assumption.
If $(f_{22},f_{02},f_{12},f_{23})=(0,0,0,0)$, then $\mathcal{C}$ is the union of four lines (since $h(X,Z) \in \bF_q[X]$) through $(0:1:0)$. Since each of these lines has $(0:1:0)$ as its unique $\bF_q$-rational point, the 
number of affine $\bF_q$-rational points on $\cC$ is one. Therefore, also the
number of affine points in $C_1(\bF_q)\cap C_2(\bF_q)$ is one.
Note that in this case the line $\ell=C_1 \cap \pi_{\infty}$ meets $C_2$ at the point $(0:0:1:0)$, so that $C_1(\bF_q)\cap C_2(\bF_q)$ has size two.

\bigskip

If $f_{22}\ne 0$ then without loss of generality we may assume $f_{22}=1$. Then 
$$h(X,Z)=(Z-\alpha(X))(Z-\beta(X)),$$
with
$$\alpha(X)=a_2X^2+a_1X+a_0\notin\mathbb{F}_q[X]$$ 
and 
$$\beta(X)=a_2^qX^2+a_1^qX+a_0^q,$$ 
for some $a_0,a_2,a_2\in \bF_{q^2}[X]$.
The point $(x,z)$ is an $\mathbb{F}_q$-rational point of $\mathcal{C}$ if and only if it is an $\mathbb{F}_q$-rational point of
\[ \left\{ \begin{array}{ll} Z=\alpha(X), \\ Z=\beta(X), \end{array}\right. \]
which can have at most four $\bF_q$-rational points.
Together with the points at infinity, $C_1(\bF_q) \cap C_2(\bF_q)$ can have at most $6$ points.
\end{proof}

\begin{theorem}\label{th:class7}
For $q\geq 53$ there are no identifiable Waring subspaces of codimension two in $\PP_q^9$ with respect to $\VV_{3,2}(\bF_q)$, i.e. 
$$\eta_7(\VV_{3,2}(\bF_{q\geq 53}))=0.$$
Furthermore, an identifiable Waring subspace of codimension two in $\PP_q^9$  with respect to $\VV_{3,2}(\bF_q)$, for $q\leq 13$, is spanned by the images under $\nu_2$ of one the following curves of $\PP_q^3$:
\begin{enumerate}
    \item four distinct lines for $q=2$;
    \item two skew lines over $\bF_q$ and two conjugate lines over $\bF_{q^2}$ for $q=3$;
    \item two disjoint conics over $\bF_q$ for $q=3$;
    \item two conics intersecting in two points for $q=4$;
    \item an elliptic curve for $q\in \{ 4,5,7,8,9,11,13 \}$;
    \item an absolutely irreducible rational curve (which is not a twisted cubic and which is a base of a linear system of quadrics of dimension two) for $q=7$.
\end{enumerate}
Moreover, each of these cases occurs and gives at least one example of an identifiable Waring subspace. 
\end{theorem}

\begin{proof}
Recall that identifiable Waring subspaces of dimension $7$ correspond to pencils of quadrics whose base consists of exactly $8$ points which is not contained in any net of quadrics of $\PP^3$. 
The pencil of quadrics defined by the quadratic forms $f$ and $g$ is denoted by $\cP(f,g)$.
We can divide the analysis of pencils of quadrics into the following cases:
\begin{enumerate}[$(i)$]
\item the pencil contains a non-singular quadric;
\item the pencil consists of singular quadrics and at least two quadratic cones;
\item the pencil consists of singular quadrics and one quadratic cone;
\item the pencil consists of reducible quadrics.
\end{enumerate}
In case $(i)$, by \cite[Sections 4 and 8]{BruenHir}, the base curve $\C$ of the pencil can be one of the following curves (for convenience of the reader, we follow the numbering from \cite{BruenHir}). By the above, we may restrict ourselves to the case where the number of $\bF_q$-rational points of $\cC$ can be even.

[1(a)] $\C$ is the union of four distinct lines over $\bF_q$, $|\C(\bF_q)|=4q$.
An example is given by $\cP(X_0X_2,X_1X_3)$, and for $q=2$ we obtain a codimension two identifiable Waring subspace.

[1(b)] $\C$ is the union of two distinct skew lines over $\bF_q$ and two conjugate lines over $\bF_{q^2}$, $|\C(\bF_q)|=2q+2$. An example is given by 
$$\cP(X_0X_3-X_1X_2,X_0X_2+b X_1 X_2 +X_1 X_3),$$
and this gives an identifiable Waring subspace for $q=3$.

[3(a)] $\C$ is the union of two disjoint conics over $\bF_q$, $|\C(\bF_q)|=2q+2$.
An example is given by 
$$\cP(X_0X_1, X_0^2+X_1^2+X_2^2+b X_2X_3+X_3^2),$$
where $b \in \mathbb{F}_q$ such that $x^2+bx+1$ is irreducible over $\mathbb{F}_q$ $b \in \mathbb{F}_q$ such that $x^2+bx+1$ is irreducible over $\mathbb{F}_q$. For $q$ odd there is another (inequivalent) example given by 
$$\cP(X_0X_1,X_0^2+\nu X_1^2+X_2^2+b X_2X_3+X_3^2),$$ 
where  and $\nu$ is a non-square in $\mathbb{F}_q$. Both pencils give identifiable Waring subspaces for $q=3$, which are inequivalent.

[3(c)] $\C$ is the union of two conics over $\bF_q$ with two points in common, $|\C(\bF_q)|=2q$.
An example is given by 
$$\cP(X_0X_1,X_0^2+X_1^2+ X_2X_3),$$  
and for $q$ odd, there is another inequivalent example given by 
$$\cP(X_0X_1,X_0^2+\nu X_1^2+ X_2X_3),$$ 
with $\nu$ a non-square in $\mathbb{F}_q$.
These pencils give inequivalent identifiable Waring subspaces for $q=4$.

[4(a)] $\C$ is the union of a twisted cubic and a bisecant, $|\C(\bF_q)|=2q$.
There is unique such pencil. An example is given by $\cP(X_0X_2-X_1^2,X_1 X_3-X_2^2)$, and this pencil induces an identifiable Waring subspace for $q=4$.

[4(c)] $\C$ is the union of a twisted cubic and an imagining chord, $|\C(\bF_q)|=2q+2$.
Again there is a unique such pencil. An example is given by
$\cP(X_0X_3-X_1 X_2+b X_1 X_3-bX_2^2,X_0 X_2 -X_1^2-X_1X_3+X_2^2)$, with  $b \in \mathbb{F}_q$ such that $x^2+bx+1$ is irreducible over $\mathbb{F}_q$. This pencil gives an identifiable Waring subspace for $q=3$.

[5(a)] $\C$ is an irreducible elliptic curve, $(\sqrt{q}-1)^2\leq |\C(\bF_q)|\leq (\sqrt{q}+1)^2$.
It is well-known that elliptic curves are the intersection of two quadrics of $\PP^3$ and hence the subspaces spanned by the images of the $\bF_q$-rational points are identifiable Waring subspaces of codimension two in $\PP^9$. 
Moreover, the existence of such subspaces is guaranteed by Theorem \ref{thm:codim_two_P9}.

[5(b)] $\C$ is an irreducible rational curve, $|\C(\bF_q)|= q+1$. Since the twisted cubic is the base curve of a linear system of quadrics in $\PP^3$ of dimension three for any $q$; see \cite[Lemma 21.1.6]{Hir3}, this case does not lead to an identifiable Waring subspace. 

Now suppose $q\geq 53$. Clearly, as we need $8$ points, we can now analyze the case when the size of the intersections of the two quadrics is less than $q+1$. This means that we only need to consider the cases $(ii)$, $(iii)$, and $(iv)$.
The case of the intersection of two quadratic cones is taken care of by Lemma \ref{lem:int2cones}.
Consider now the intersection between a pair of planes $\pi_1 \cup \pi_2$ and a cone $C$ such that $|C\cap (\pi_1\cup\pi_2)|<q+1$. Then $|C\cap (\pi_1\cup\pi_2)|\leq 1$, since a plane $\pi$ meets a cone in either zero points, one point (the vertex of the cone) or $q+1$ points.
As we need $8$ points, this case does not give rise to any identifiable Waring subspaces.
If the two quadrics are two pairs of planes, then their intersection is at least $q+1$ points. Again, this case does not give rise to any identifiable Waring subspaces.
\end{proof}

\begin{remark}
We have verified computationally, that there are no identifiable Waring subspaces of codimension two in $\PP_{16}^9$.
\end{remark}

\begin{corollary}
The coefficient $\eta_7(\mathbb{V}_{3,2}(\F_q))$ of the identifiable Waring polynomial of $\mathbb{V}_{3,2}(\F_q)$ is: 
\begin{itemize}
\item one if $q =2$;
\item three if $q=3$;
\item at least $2$ if $q=4$;
\item at least one if $q\in \{5,7,8,9,11,13\}$.
\end{itemize}
\end{corollary}
\begin{proof}
It is clear that two identifiable Waring subspaces of codimension two in $\PP_q^9$ obtained by curves from two different cases (1)-(5) of Theorem \ref{th:class7} are not equivalent.
By \cite[Theorem 4.4]{BruenHir}, the number of equivalence classes arising from the examples in (1),(2) and (4) of Theorem \ref{th:class7} is one, whereas the number of equivalence classes arising from the examples in (3) of Theorem \ref{th:class7} is two.
The assertion is then proved.
\end{proof}

\subsection{Identifiable Waring hyperplanes}
Since the hyperplanes sections of $\bP_q^9$ correspond to quadrics of $\PP_q^3$, we have the following.

\begin{theorem}\label{th:classhyper3}
The identifiable Waring hyperplanes of $\PP_q^9$ exist only for $q=2$ and they are the span of the images under $\nu_2$ of hyperbolic quadrics in $\PP_q^3$. In particular, $\eta_8(\VV_{3,2}(\bF_2))=1$ and $\eta_8(\VV_{3,2}(\bF_{q>2}))=0$.
\end{theorem}
\begin{proof}
In order to obtain an identifiable Waring hyperplane, one needs a quadric of $\bP_q^3$ consisting of exactly nine points, with the property that its image under the Veronese map spans a hyperplane. A quadric in $\bP_q^3$ is one of the following:\begin{itemize}
  \item a hyperbolic quadric (i.e. $(q+1)^2$ points);
  \item an elliptic quadric (i.e. $q^2+1$ points);
  \item an irreducible quadratic cone (i.e. $q^2+q+1$ points);
  \item a pair of conjugate planes over $\bF_{q^2}$ (i.e. $q+1$ points)
  \item two distinct planes (i.e. $2q^2+q+1$ points);
  \item one plane (i.e. $q^2+q+1$ points).
\end{itemize}

The only possibility is a hyperbolic quadric in $\bP_2^3$, and there is a unique such quadric up to projective equivalence.
\end{proof}

\section{Waring polynomials for elliptic and hyperbolic quadrics in $\bP^3(\bF_q)$}

In this section we collect some of the well-known properties of the elliptic and hyperbolic quadrics in $\bP(\bF_q)$ in terms of their Waring polynomials.

Let $\cX=\mathcal{E}$ be an elliptic quadric in $\bP^3_q$. The points on $\mathcal{E}(\bF_q)$ are Waring subspaces, and the group of the quadric acts transitively on them. When $q=2$, the quadric $\mathcal{E}(\bF_q)$ consists of 5 points.
For any external point $P$ of the quadric there exist two tangent planes and one $3$-secant plane (where the three points of intersection are not collinear). Then if $P$ is not the nuclei of the conic contained in the secant plane it will have $\mathcal{E}$-rank two, otherwise it will have $\mathcal{E}$-rank $3$.
A line $\ell$ is a Waring subspace with respect to $\mathcal{E}$ if and only if $|\ell \cap \mathcal{E}(\bF_q)|=2$ and the other lines are neither Waring subspaces nor Waring identifiable subspaces with respect to $\mathcal{E}$, unless $q=2$, in which case each external line is Waring identifiable, since it is contained in a unique plane of $\bP_2^3$ meeting $\mathcal E$ in a non-degenerate conic.
As the planes of $\PP_q^3$ meets $\mathcal{E}(\bF_q)$ in either an irreducible conic or in a point, when $q\geq3$ then the planes meeting $\mathcal{E}(\bF_q)$ in more than one point are Waring subspaces but not Waring identifiable subspaces, whereas when $q=2$ the Waring planes are also identifiable Waring planes. The tangent planes are neither Waring subspace nor Waring identifiable subspace with respect to $\mathcal{E}$.
Therefore,
if $q\ne 2$, then
\[ \mathcal{W}_{\mathcal{E}}(\mathbb{F}_q)=1+X+X^2,\,\,\mathcal{WI}_{\mathcal{E}}(\mathbb{F}_q)=1+X,\, \mathcal{IW}_{\mathcal{E}}(\mathbb{F}_q)=1+X, \]
whereas if $q=2$
\[ \mathcal{W}_{\mathcal{E}}(\mathbb{F}_2)=1+X+X^2,\,\,\mathcal{WI}_{\mathcal{E}}(\mathbb{F}_2)=3+2X+X^2,\, \mathcal{IW}_{\mathcal{E}}(\mathbb{F}_2)=1+X+X^2. \]
When $\cX=\mathcal{H}$ is a hyperbolic quadric, then the points of $\mathcal{H}(\bF_q)$ are Waring subspaces, through each external point there are at least two $2$-secant lines, hence they are not Waring identifiable. 
The only Waring lines are the $2$-secant lines (which are Waring identifiable) and the $(q+1)$-secant lines of $\mathcal{H}(\bF_q)$ (which are not Waring identifiable).
Identifiable Waring planes exist if and only if $q=2$ and are those meeting $\mathcal{H}(\bF_q)$ in an irreducible conic.
When $q\geq 3$, there do not exist identifiable Waring planes but there are two types of Waring planes: those meeting $\mathcal{H}(\bF_q)$ in an irreducible conic and those meeting $\mathcal{H}(\bF_q)$ is two lines. 
These two type of planes are inequivalent. 
Therefore, 
if $q\ne 2$, then
\[ \mathcal{W}_{\mathcal{H}}(\mathbb{F}_q)=1+2X+2X^2,\,\,\mathcal{WI}_{\mathcal{H}}(\mathbb{F}_q)=1+X,\, \mathcal{IW}_{\mathcal{H}}(\mathbb{F}_q)=1+X, \]
whereas if $q=2$
\[ \mathcal{W}_{\mathcal{H}}(\mathbb{F}_2)=1+2X+2X^2,\,\,\mathcal{WI}_{\mathcal{H}}(\mathbb{F}_2)=1+X+X^2,\, \mathcal{IW}_{\mathcal{H}}(\mathbb{F}_2)=1+X+X^2. \]

\section{Conclusions}

In this paper we introduced the concepts of (identifiable) Waring subspaces and Waring identifiable subspaces with respect to an algebraic variety. The number of projectively inequivalent such subspaces is captured with the relevant Waring polynomials.
We have seen how examples easily carry out when considering conics or quadrics over finite fields. Most of the paper is devoted to study the case when the involved algebraic variety is the Veronese variety, motivated by the connection with the theory of symmetric tensors and in particular with the Waring identifiability of symmetric tensors.
In particular, we have given examples of Waring subspaces with respect to the quadric Veronese variety and we have obtained a complete classification in $\PP^5$. Constructions and classification results are also obtained in $\PP^{9}$, but in this case a complete classification seems more difficult and the subject of possible future work. 

Another interesting problem is whether or not the Waring identifiability cases of \cite[Section 1]{AGMO} carry through over finite fields.
Since over finite fields the notion of \emph{generic form}/\emph{generic point} cannot be defined, our aim will be to find the parameters $(r,n,d_1,\ldots,d_r,k)$ for which there exists at least one $r$-tuple of forms in $\mathbb{F}_q[x_1,\ldots,x_n]_{d_1}\times\mathbb{F}_q[x_1,\ldots,x_n]_{d_r}$ which is Waring identifiable.
We remark that there are examples which still hold over finite fields, for instance the following one which is the one shown by Sylvester in \cite{Sy} over the complex field.

\begin{proposition}\label{prop:fin}
Let $t$ be a positive integer and $q\geq 2t+1$. Every point of $\PP^{2t+1}$ having $\VV_{1,2t+1}$-rank $t+1$ is Waring identifiable.
\end{proposition}
\begin{proof}
As $\VV_{1,2t+1}$ is a normal rational curve of $\PP^{2t+1}$ and hence an arc in $\PP^{2t+1}$, any two $(t+1)$-dimensional Waring subspaces with respect to $\VV_{1,2t+1}$ are disjoint or meet in at least one point on $\VV_{1,2t+1}$.
Therefore, through any point $P$ of $\PP^{2t+1}$ having $\VV_{1,2t+1}$-rank $t+1$ there exists only one $(t+1)$-dimensional Waring subspace with respect to $\VV_{1,2t+1}$ and hence $P$ is Waring identifiable.
\end{proof}

The following corollary extends the known examples of Waring identifiability of \cite[Section 1]{AGMO}.

\begin{corollary}
If $q\geq 2t+1$, there exist forms in $\mathbb{F}_q[x_1,x_2]_{2t+1}$ of rank $t+1$ which are Waring identifiable.
\end{corollary}

Note that the proof of Proposition \ref{prop:fin} may be adapted to any arc in a projective space of odd dimension. 
An example of such an arc is the following: 
let $h$ and $e$ two positive integers with $\gcd(h,e)=1$, let $q=2^h$ and let $\sigma\colon x \in \bF_q\mapsto x^{2^e}$ then the set
\[ \{(1,t,t^\sigma,t^{\sigma+1}) \colon t \in \bF_q\} \cup \{(0,0,0,1)\} \]
is an arc of $q+1$ points in $\PP^3$ known as the \emph{Segre arc}, \cite{Segre} and see also \cite{BL}.

\section{Acknowledgement}
The authors are thankful to Daniele Bartoli for valuable discussions and ideas related to Lemma \ref{lem:int2cones}.
The second author is very grateful for the hospitality of the Sabanc\i{} University, where he was a visiting researcher during the development of this research in October 2018 and November 2019.

\bigskip

\par\noindent Michel Lavrauw\\
Sabanc\i{} University, Istanbul, Turkey\\
{{\em mlavrauw@sabanciuniv.edu}}

\bigskip
\par\noindent Ferdinando Zullo\\
Universit\`a degli Studi della Campania ``Luigi Vanvitelli'', Caserta, Italy\\
{{\em ferdinando.zullo@unicampania.it}}

\end{document}